\documentclass[11pt]{article}
\usepackage{geometry}
\usepackage[utf8]{inputenc}
\usepackage{amsmath,amsthm,amssymb,color,graphicx,url,hyperref,diagbox,enumerate,cite,tikz,authblk}
\usepackage{appendix}
\usepackage{makecell}

\newtheorem{theorem}{Theorem}

\newtheorem{problem}[theorem]{Problem}

\newtheorem{remark}[theorem]{Remark}

\newcounter{encoding}
\newtheorem{encoding}[encoding]{Encoding}

\newcommand{\Z}{\mathbb{Z}}

\newcommand{\inv}{^{-1}}

\usepackage{graphicx}

\title{New bounds for some small multicolor Ramsey numbers}
\author{William J. Wesley}
\date{\today}

\begin{document}

\maketitle

\begin{abstract}
    The Ramsey number $R(G_1,\dots,G_k)$ is the smallest $n$ such that every $k$-coloring of the edges of $K_n$ contains a monochromatic copy of $G_i$ in color $i$. Ramsey numbers are challenging to compute, and few are known exactly. We use Boolean satisfiability (SAT) solvers to search for structured colorings that give lower bounds, and we show $R(K_4,K_4-e,K_4-e) \ge 35$ and $R(K_3,K_4,C_4,C_4) \ge 49$. Moreover, we tighten some recent upper bounds for multicolor Ramsey numbers for cycles and show $R(C_3,C_6,C_6) = R(C_5,C_6,C_6) = 15$. Finally, we enumerate critical graphs for the numbers $R(C_4,K_{1,s})$ and $R(C_6,K_{1,s})$. 
\end{abstract}

\section{Introduction}

The \emph{multicolor Ramsey number} $R(G_1,\dots,G_k)$ is the smallest $n$ such that every $k$-coloring of the edges of $K_n$ contains a monochromatic copy of $G_i$ in color $i$. The most famous are the \emph{classical} Ramsey numbers where each $G_i$ is a complete graph $K_{n_i}$, and here we write simply $R(n_1,n_2,\dots, n_k)$. The two color cases have been studied extensively, and enormous computational and theoretical effort has gone into determining both small values and asymptotics. Despite this, we know few nontrivial values for these Ramsey numbers: the only exact values that have been determined are $R(3,t)$ for $3 \le t \le 9$, $R(4,4)$, and $R(4,5)$. The bounds on $R(5,5)$ stand at $43 \le R(5,5) \le 46$, and determining the precise value remains a notoriously difficult open problem \cite{ExooR55LowerBound,R55Le46,R45McKayRadzidzowksi}.

The three color classical case is even harder. Here we know only two values: $R(3,3,3) = 17$ from Greenwood and Gleason in 1955 \cite{R333Equals17} and $R(3,3,4) = 30$ from Codish, Frank, Itzhakov, and Miller in 2016 \cite{R334Equals30}. This latter result required extensive computation. Only a handful of other Ramsey numbers involving other graphs three or more colors are known, and we refer the reader to the comprehensive survey of Radziszowski \cite{RamseySurvey} for more details. 

While the best asymptotic lower bounds for (classical, two color) Ramsey numbers come from random colorings \cite{SpencerRamseyLB}, the lower bounds for small Ramsey numbers often come from highly symmetric and structured colorings. One family of these colorings is the \emph{Cayley colorings}, which arise from groups. More precisely, let $\Gamma$ be a group of order $n$ with identity $e$, and label the vertices of $K_n$ by the elements of $\Gamma$. A $k$-coloring $\chi$ of the edges of $K_n$ is a \emph{Cayley coloring} if for each color $c$ there exists a set $S_c \subseteq \Gamma\setminus \{e\}$ such that the edge $\{x,y\}$ is color $c$ if and only if $xy\inv \in S_c$. If the group $\Gamma$ is cyclic, then we say that $\chi$ is a \emph{circulant} coloring. Circulant colorings give good lower bounds for several small classical Ramsey numbers: there are circulant colorings for $R(s,t)$ for $(s,t) = (3,3), (3,4), (3,5), (4,4), (4,5)$, and $(3,9)$ that give tight lower bounds, though this is not the case in general \cite{HarborthKrauseCirculant,HarborthKrauseDistance}. 

Two natural generalizations of circulant and Cayley colorings are the \emph{block circulant} and \emph{block Cayley colorings}, which we define precisely in Section \ref{SectionPreliminaries}. Block circulant colorings were used by Exoo and Tatarevic to establish several classical Ramsey number lower bounds \cite{ExooTaterevic28Classical}. They were recently studied in \cite{BlockCircRamseyGoedVanOver, LidickyMcKinleyPfenderSmallBooksWheels} and \cite{WJWBooks} and produced new lower bounds for Ramsey numbers of the form $R(K_s,K_t-e)$ and $R(B_s,B_t)$, where $K_t-e$ is the graph obtained by removing one edge from the complete graph $K_t$, and $B_s$ is the \emph{book graph} $K_2 + \overline{K}_s$ that consists of $s$ triangles sharing a common edge. Block Cayley colorings have appeared implicitly in the literature as well. A famous coloring of Chung gives the best known lower bound of 51 for $R(3,3,3,3)$: this comes from a block Cayley 4-coloring with 3 blocks of size 16 that is extended by two additional vertices \cite{ChungR3333LB}. 

One goal of this project was to improve lower bounds for other multicolor Ramsey numbers by searching for structured colorings. There are many ways to search for interesting graphs and colorings, but we used Boolean satisfiability (SAT) solving. SAT solvers are powerful tools in computational combinatorics, and they have already been used successfully in improving Ramsey number bounds \cite{R55Le46,R334Equals30,WJWBooks,LowKapKapBereg,DybizbanskiHypergraphLB}. They also are able to certify nonexistence of colorings, which is useful for eliminating certain groups for block Cayley colorings. When no additional constraints are imposed to search for special colorings, a proof of nonexistence of a coloring gives an \emph{upper bound} for the Ramsey number. In general, certifying a Ramsey number upper bound with a SAT solver is more difficult than finding one coloring to prove a lower bound, but it can nonetheless be done in small cases. 

\subsection{Contributions}
We targeted multicolor Ramsey numbers of small to moderate size for various graphs. Our first result is an improvement on the lower bound for $R(K_4,K_4-e,K_4-e)$. The previous lower bound $R(K_4,K_4-e,K_4-e) \ge 33$ came from a coloring of $K_{32}$ found by Shetler, Wurtz, and Radziszowski using simulated annealing \cite{ShetlerWurtzRadz}, and the upper bound stands at 47 \cite{LidickyPfenderSDPRamsey}. 

\begin{theorem} \label{K4J4J4lb}
    $R(K_4,K_4-e,K_4-e) \ge 35$.
\end{theorem}

We prove this by exhibiting a 3-coloring of $K_{34}$ that avoids $K_4$ in the first color and $K_4 - e$ in the other two colors. This is similar to Chung's $R(3,3,3,3)$ construction in that it consists of a block Cayley coloring extended by two vertices. We give the details of our coloring in Section \ref{SectionConstructions}. 

The Ramsey number $R(K_3,K_4,C_4,C_4)$ was first studied by Xu, Shao, and Radziszowski, who gave a lower bound of 42 \cite{XuShaoRadzRamseyQuadrilateral}. This arose from a construction involving a one-vertex extension of a product coloring of critical colorings for $R(K_3,K_4) = 9$ and $R(C_4,C_4) = 6$. Later, Dybizba\'nski and Dzido improved the lower bound to 43 with a coloring found by computer \cite{DyDzRamseyQuad}. Our result improves the lower bound to 49. 

\begin{theorem} \label{K3K4C4C4lb}
    $R(K_3,K_4,C_4,C_4) \ge 49$.
\end{theorem}

The upper bound is still relatively far away, but it was recently improved to 75 by Boza and Radziszowski \cite{boza_radziszowski_2025}.
 For other upper bounds, we make progress on 3-color Ramsey numbers involving cycles. Many bounds and exact values for these numbers were determined by Sun, Yang, Wang, Li, and Xu in \cite{SunYangWangLiXu3ColorCycles} and Tse in \cite{Tse3ColorCycles} (and others, see \cite{RamseySurvey}). Lidick\'y and Pfender proved the upper bounds $R(C_3,C_6,C_6) \le 18$ and $R(C_5,C_6,C_6) \le 17$ using flag algebras and semidefinite programming \cite{LidickyPfenderSDPRamsey}. We tighten these bounds. 

\begin{theorem} \label{CycleUBs}
    $R(C_3,C_6,C_6) = 15$.
\end{theorem}
\begin{theorem} \label{CycleUBs2}
    $R(C_5,C_6,C_6) = 15$.
\end{theorem}

The \emph{hypergraph Ramsey number} $R(G_1,\dots,G_k;r)$ is the smallest $n$ such that every $k$-coloring of the $r$-subsets of $[n]$ contains a monochromatic copy of the hypergraph $G_i$ in color $i$. Let $K_4^-$ denote the 3-uniform hypergraph on 4 vertices where all but one of the four 3-subsets of $[4]$ are hyperedges. A coloring showing a lower bound of 13 was found by Exoo \cite{ExooHypergraphs} and the upper by Lidick\'y and Pfender stood at 14 \cite{LidickyPfenderSDPRamsey}. The following result closes this gap.

\begin{theorem} \label{HypergraphUBs}
$R(K_4^-,K_4^-,K_4^-;3) = 13$
\end{theorem}

Though finding a single coloring is enough to prove a Ramsey number lower bound, an interesting question is to enumerate \emph{all} the colorings (graphs) for given Ramsey numbers up to isomorphism. We say that a coloring (graph) is a \emph{critical} coloring (graph) for a Ramsey number $n = R(G_1,\dots,G_k)$ if it is a $k$-coloring of $K_{n-1}$ that does not contain a copy of $G_i$ in color $i$. As mentioned above, often these graphs have special properties or symmetries. We used SAT modulo symmetries (SMS), a method that leverages a SAT solver to generate graphs with given properties, to enumerate critical graphs for the Ramsey numbers $R(C_4,K_{1,s})$ and $R(C_6,K_{1,s})$ \cite{SMS} . These numbers were convenient because relatively large tables of their values exist \cite{RamseySurvey}. Of particular interest is the fact that the critical graphs for $R(C_4,K_{1,12})$ and $R(C_4,K_{1,28})$ are unique. The latter is the unique (6,3)-cage (that is, a 6-regular graph of girth 3) that does not contain 4-cycles \cite{EzeJajcayJookenCages}. 

\begin{theorem} \label{TheoremSMS} The following critical graph counts hold.

\begin{center}
\begin{tabular}{|c|c|c|}
\hline
    $s$ & $R(C_4,K_{1,s})$ & \# critical graphs  \\
    \hline
    3 & 6 & 2 \\
    4 & 7 & 4 \\ 
    5 & 8 & 10\\
    6 & 9 & 30\\
     7 & 11 & 5 \\
     8 & 12& 9\\
     9 & 13 & 57 \\
     10 & 14 & 503 \\
     11 & 16 & 2 \\
     12 & 17 & 1 \\
     13 & 18 & 8 \\
     14 & 19 & 132 \\
     15 & 20 & 4181 \\
     16 & 21 & 195579 \\
     21 & 27 & 17 \\
     22 & 28 & 353\\
     28 & 35 & 1  \\
     29 & 36 & 6 \\
      \hline
\end{tabular}
\begin{tabular}{|c|c|c|}
\hline
    $s$ & $R(C_6,K_{1,s})$ & \# critical graphs  \\
    \hline
    4 & 8 & 1 \\ 
    5 & 9 & 1  \\
    6 & 11 & 2 \\
    7 & 11 & 38 \\
    8 & 12 & 54 \\
    9 & 14 & 4 \\
    10 & 15 & 7\\
    11 & 16 & 10\\
    12 & 17 & 2\\ 
      \hline
\end{tabular}
\end{center} 
\end{theorem}

\section{Definitions and Methods}\label{SectionPreliminaries}
This section gives the complete definitions of the structure of the colorings (block circulant and block Cayley) that we searched for in order to improve lower bounds. We also give additional details on how to encode this into a SAT formula as well as other solving techniques we used.  

\subsection{Block circulant and block Cayley colorings}

Let $\Gamma$ be a group. A \emph{Cayley graph} is a graph $G$ with vertex set $\Gamma$ such that there exists a set $S \subseteq \Gamma$ such that for all $x,y \in V = \Gamma$, we have that $\{x,y\}$ is an edge if and only if $xy\inv \in S$. The set $S$ cannot contain the identity of $\Gamma$ if $G$ is loopless, and it must be closed under inverses if edges are undirected. We can also consider \emph{directed} Cayley graphs where $S$ need not be closed under inverses.  

A coloring of $K_n$ is \emph{Cayley} with respect to a group $\Gamma$ of order $n$ in color $i$ if the subgraph induced by color $i$ is a Cayley graph of $\Gamma$. If $\Gamma = \Z_n$, we say the coloring is \emph{circulant}. A coloring of $K_{kn}$ is \emph{block Cayley} with $k$ blocks with respect to a group $\Gamma$ of order $n$ in color $i$ if the adjacency matrix $A$ of the graph induced by color $i$ is of the form $$A = \begin{pmatrix}C_{11} & C_{12} & \dots & C_{1k} \\ C_{21}  & C_{22} & \dots & C_{2k} \\ \vdots & \vdots & \ddots & \vdots 
& \\ C_{k1} & C_{k2} & \dots & C_{kk}\end{pmatrix}, $$ where $C_{ij}$ is 
the matrix of a directed Cayley graph (on $\Gamma$) for all $i,j$. If $\Gamma = \Z_n$, we say the coloring is \emph{block circulant}. If a coloring is block Cayley (circulant) in all colors, we say it is block Cayley (circulant).


\subsection{SAT encoding}

The following is a natural encoding for determining whether $R(G_1,\dots, G_k) > n$. Given a complete graph $K_n$, for each edge $e$ and color $c$, we have a Boolean variable $x_{e,c}$ that is assigned true if and only if $e$ is assigned color $c$. For each edge $e$, we want it to be colored exactly one color, so we include the following two formulas $P_e$ and $O_e$:

\begin{align*}
    P_e := x_{e,1} \vee x_{e,2} \vee \cdots \vee x_{e,k}, \quad 
    O_e := \bigwedge_{1\le c_1 < c_2 \le k} (\bar{x}_{e,c_1} \vee \bar{x}_{e,c_2}).
\end{align*}

The formula $P_e$ states that the edge $e$ is colored at least one color, and the formula $O_e$ states that $e$ is colored at most one color. The latter formula is not strictly necessary, but it ensures that there is a one-to-one correspondence between satisfying assignments and colorings. Then for each subgraph $H$ of $K_n$ isomorphic to $G_i$, we need to include a clause that forbids all the edges of $H$ being assigned color $i$. The clause $N_{H,i}$ accomplishes this: 
\begin{align*}
N_{H,i} :=  \bigvee_{e\in E(H)} \bar x_{e,i}.
\end{align*}
Putting it all together, we have the following encoding. 
\begin{encoding}\label{EncodingRamseySAT}
	The Ramsey number $R(G_1,G_2,\dots, G_k)$ is at most $n$ if the formula $F_n(G_1,G_2,\dots, G_k)$ is unsatisfiable, where 
	$$F_n(G_1,G_2,\dots, G_k) := \Bigl(\bigwedge_{e \in E(K_n)} P_e \Bigr) \wedge \Bigl(\bigwedge_{e \in E(K_n)} O_e \Bigr) \wedge\Bigl(\bigwedge_{i=1}^k\bigwedge_{H\subset K_n, H\cong G_i} N_{H,i} \Bigr).$$
	Moreover, if $F_n(G_1,G_2,\dots, G_k)$ is satisfiable, then $R(G_1,G_2,\dots, G_k) \ge n+1$.
\end{encoding}
\begin{remark}
    For Ramsey numbers involving hypergraphs, the encoding essentially the same. We have variables for the possible colors assigned to hyperedges, ensure each hyperedge is assigned exactly one color, and we forbid monochromatic copies of hypergraphs in the same way. 
\end{remark}

The formulas $F_n(G_1,\dots,G_k)$ are in conjunctive normal form (CNF), which is the standard input to SAT solvers. Unfortunately, for many Ramsey numbers, these formulas alone are not sufficient to compute good bounds in a reasonable amount of time. To obtain better lower bounds, for example, we need to restrict the search space. Given a group $\Gamma$ of order $n$, we can add the following clauses to enforce that a coloring of $K_{bn}$ is block Cayley with $b$ blocks in a given color $i$:

\begin{align*}
    C_{\Gamma,i} := \bigwedge_{0 \le b_1 \le b_2 < b} \bigwedge_{g \in \Gamma\setminus \{e\}} \bigwedge_{h_1,h_2 \in \Gamma} (\bar{x}_{\{h_1+b_1n, g*h_1+b_2n\},i} \vee x_{\{h_2+b_1n, g*h_2+b_2n\},i}) 
\end{align*}

In the above formula the vertex set of $K_{bn}$ is $\{0,\dots,bn-1\}$ and the blocks are labeled 0 to $b-1$. We are considering the underlying set of $\Gamma$ to be $\{0,\dots,n-1\}$ and denoting the group operation by $*$. The clauses have the effect that in each pair of blocks $(b_1,b_2)$, there is a subset $S = S_{b_1,b_2}$ of $\Gamma \setminus \{e\}$ that determines which edges are assigned color $i$. It is possible to modify Encoding \ref{EncodingRamseySAT} directly instead and use fewer variables, having for instance for each pair of blocks $(b_1,b_2)$ a variable $x_{g,b_1,b_2,i}$ that is true if and only if all edges $\{x+b_1n,y+b_2n\}$ with $x*y\inv = g$ are assigned color $i$. However, in practice this doesn't matter because the solver we used, Kissat, performs this simplification automatically in preprocessing \cite{Kissat}. 

Sometimes we can extend Cayley colorings by one or more vertices. One can show that $R(3,3,3,3) \ge 51$, for example, with is a 3-block Cayley coloring using the group SmallGroup(16,2), which is isomorphic to $\Z_4 \times \Z_4$, extended by 2 vertices. If we want to extend a block Cayley coloring with a group of order $n$ and $b$ blocks by $c$ vertices, this can be done in the SAT encoding by using Encoding $\ref{EncodingRamseySAT}$ for all $bn+c$ vertices and simply including the block Cayley clauses for the induced coloring on the first $bn$ vertices. 

The addition of the block Cayley clauses to $F_{bn}(G_1,\dots, G_k)$ results in a dramatic reduction in the search space, and satisfying assignments to the new formula are still satisfying assignments to the original. However, we cannot add block Cayley clauses to produce upper bounds for Ramsey numbers since there may exist colorings that are not block Cayley. Here we must introduce \emph{symmetry breaking} clauses which do not affect the satisfiability of the formula, but prevent the solver from wasting time checking isomorphic assignments. The main symmetry we are interested in breaking is the labeling of the vertices. We do this with the automated symmetry breaking software Shatter \cite{Shatter}. This was sufficient for proving the upper bounds in Theorems \ref{CycleUBs}, \ref{CycleUBs2} and \ref{HypergraphUBs}. 

\subsection{SAT modulo symmetries}

SAT modulo symmetries (SMS) is a framework to enumerate graphs with a SAT solver and \emph{dynamic} symmetry breaking, wherein intermediate assignments are checked for lexicographic minimality and isomorphic assignments are filtered. The software \cite{SMS} can be used to generate graphs with given constraints (for example, number of vertices, number of edges, bounded degrees, chromatic number, etc.). It allows also for custom constraints to be entered with a CNF formula. For instance, we can forbid 4-cycles in a graph by adding all clauses of the from $\bar{x}_{e_1} \vee \bar{x}_{e_2} \vee \bar{x}_{e_3} \vee \bar{x}_{e_4}$, where $e_1,e_2,e_3,e_4$ are edges in a 4-cycle. 

\section{Constructions and Proofs}\label{SectionConstructions}

In this section we prove our main results and discuss our computations. All SAT solving was done with the solver Kissat \cite{Kissat}, and all Cayley tables and other group data were acquired through GAP \cite{GAP4}. 

\subsection{Constructions for lower bounds}
The proofs of Theorem \ref{K4J4J4lb} and Theorem \ref{K3K4C4C4lb} come from explicit colorings. It is straightforward to verify via computer that these colorings have the desired properties, i.e., they avoid subgraphs of particular colors. 

\begin{proof}[Proof of Theorem \ref{K4J4J4lb}]
    The following edge coloring of $K_{34}$ does not contain a copy of $K_4$ in color 2 or $K_4-e$ in colors 0 and 1. 
\begin{align*}
x121210100210122021110012012202220\\[-5pt]
1x11102102021220012101001001220011\\[-5pt]
21x2220111000122122000122112001200\\[-5pt]
112x022122011201212200011201110220\\[-5pt]
2120x10021210220121122212020100000\\[-5pt]
10221x1021212221021011102122110020\\[-5pt]
020201x202120120211101202201020222\\[-5pt]
1111002x22012120110220011200001001\\[-5pt]
00122202x0122112222022100112200011\\[-5pt]
021211220x021022122002012221000021\\[-5pt]
2000221010x22111201021101121220020\\[-5pt]
12011121222x1100122120011201221021\\[-5pt]
010102022121x200221121002210212021\\[-5pt]
1212221110112x00021122012120121121\\[-5pt]
22202222121000x0221000201022020011\\[-5pt]
202101002210000x121201211012200212\\[-5pt]
0012102121212021x11022120000122020\\[-5pt]
21212211220222221x2210201011022221\\[-5pt]
122211102212111112x122100001111221\\[-5pt]
1102101200011102021x20002120112120\\[-5pt]
10002102202222002122x2001022120101\\[-5pt]
010021102210120120202x001211221022\\[-5pt]
0010212010100022121000x01211122211\\[-5pt]
10211001010101012000000x1222122022\\[-5pt]
212122210211221101021111x200022200\\[-5pt]
0012012212122100000102222x02021222\\[-5pt]
10102200122012210102211200x1112221\\[-5pt]
212102102111002201102112021x112102\\[-5pt]
2201110020222102101112110011x00022\\[-5pt]
02010120002212202211222222110x2000\\[-5pt]
201000010001210022120122212202x222\\[-5pt]
2022002000000102022110202221002x02\\[-5pt]
21020220122222112222021202202020x2\\[-5pt]
010000211101111201101212021220222x\\[-5pt]
\end{align*}
\end{proof}

\begin{proof}[Proof of Theorem \ref{K3K4C4C4lb}]
    The following edge coloring of $K_{48}$ does not contain a copy of $K_3$ in color 3, $K_4$ in color 2, or $C_4$ in colors 0 and 1. 
\begin{align*}
    x23311302302303233233201321232030233032223223233\\[-5pt]
2x0332033233033023321222123322332003113321233132\\[-5pt]
30x223023211232231320223221330232322223312033233\\[-5pt]
332x13320330202322233233203321233222120322322132\\[-5pt]
1321x3223323210320231133322322332232320023122332\\[-5pt]
12333x122132322232232022321231233032223313222222\\[-5pt]
300321x23221233233133213321033120202322202320322\\[-5pt]
0322222x2323023323202332222202332233210331232233\\[-5pt]
23303232x231122202322232221232332223332322231232\\[-5pt]
322331232x10232103223320203303202332322212223022\\[-5pt]
0313232231x2332021232232322223302133202222232120\\[-5pt]
23103213102x331330332133121232221033202232222222\\[-5pt]
302223201233x22233023321233223013303202322221322\\[-5pt]
0330123223332x3220323223323212332212332231223222\\[-5pt]
33220233222123x202023322233220213223223333223232\\[-5pt]
202332232103222x11230223320220321132332232233223\\[-5pt]
3232233200233201x2322223322233200202231222332002\\[-5pt]
33120233231030212x032322233032222222322023321222\\[-5pt]
233222123223030230x32022123311232132223230233232\\[-5pt]
3223333022332223233x1233220033123202322223121222\\[-5pt]
31031232232233302221x223302133222223223023333202\\[-5pt]
222210232321323223022x33202221232223222123233222\\[-5pt]
0223321332332222222323x3332103322322333132123323\\[-5pt]
12333232202313233223333x222022231222332120232102\\[-5pt]
312233322231232332123232x22332013322222331233212\\[-5pt]
2220222220223232232200322x3322310331321232220033\\[-5pt]
13132112132133302330222223x202322212222321232222\\[-5pt]
233332022322222220301210332x23233212333202332231\\[-5pt]
3232233030232122331332023202x0333213222322221212\\[-5pt]
22012132233232003213313222230x332223322120322322\\[-5pt]
032232133232032322212232033233x33332123321233023\\[-5pt]
3333332330021312023223231123333x2322223322233211\\[-5pt]
02232302222132310223222130233232x212230222332112\\[-5pt]
203220222310322122122232332222332x03201233223322\\[-5pt]
3022330323330123023022222311123210x3020222223312\\[-5pt]
33222223323332322222332221223322233x133222223223\\[-5pt]
012132323322232323232233232323122201x12333320233\\[-5pt]
3122222132000323322222332223222230231x2003320320\\[-5pt]
23300320222222321232323221232233010322x222232023\\[-5pt]
233303233222323220220111323231332222302x22222323\\[-5pt]
2212210321232333223222323320222223223022x3202331\\[-5pt]
31223321222221322303332012122012232233223x321233\\[-5pt]
220312322222222233213212222323223222332223x22310\\[-5pt]
2332222332322223323233233233223332222232022x3332\\[-5pt]
33322202132213332131333230221233233300222123x032\\[-5pt]
212132322012322202222231202223021332230332330x01\\[-5pt]
3333322332222232023202201323122112123222331330x3\\[-5pt]
32322223220222232222223223212231222330331302213x\\[-5pt]
\end{align*}
\end{proof}
We found these colorings via Kissat and Encoding \ref{EncodingRamseySAT}. Both solve times were fast (6 seconds for the first coloring, 27 seconds for the second). For the first coloring, we included constraints to enforce that the coloring was block Cayley on the first 32 vertices using the group SmallGroup(16,6) in GAP. For the second we enforced the Cayley constraints on the $K_3$ and $K_4$ colors only. The group we chose was
SmallGroup(48,4) in GAP, which is isomorphic to $\Z_8 \times S_3$. We also found colorings using other groups, namely SmallGroup(48,$i$) for $i$ = 4,11,12,16,18,19,21,26,27,34,35,37,38,39,40,42,43,45. None of these groups are abelian. We also found colorings that were two block Cayley in the $K_3$ and $K_4$ colors for SmallGroup(24,$i$) for $i = 4,7$, and 3 block Cayley for SmallGroup(16,$i$) for $i = 1,2,4,10,12,13$. For each, we tried extending the coloring by one vertex for at least an hour, but we were unsuccessful. 

\subsection{SAT-based proofs}

\begin{proof}[Proofs of Theorems \ref{CycleUBs}, \ref{CycleUBs2}, and \ref{HypergraphUBs}]

For each upper bound, we used Encoding \ref{EncodingRamseySAT} to obtain formulas $F_{15}(C_3,C_6,C_6)$, $F_{15}(C_5,C_6,C_6)$, and $F_{13}(K_4^-,K_4^-,K_4^-)$, where $K_4^-$ denotes the 3-uniform hypergraph on 4 vertices with all but one hyperedge present. For each we added symmetry breaking clauses using Shatter and ran the solver Kissat. All formulas were unsatisfiable. 
\end{proof}

The formula $F_{13}(K_4^-,K_4^-,K_4^-)$ was shown to be unsatisfiable by Kissat in 645 seconds. The others were somewhat longer: the formula $F_{15}(C_3,C_6,C_6)$ took 223 minutes, and $F_{15}(C_5,C_6,C_6)$ took 142 minutes. 

\begin{proof}[Proof of Theorem \ref{TheoremSMS}]
For each case, we used the SMS software \cite{SMS} and generated constraints to find graphs of bounded degree (so that they do not contain a copy of $K_{1,s}$) that do not contain cycles of the given length.  
\end{proof}
The successful computations here were relatively fast, with the longest being the enumeration of the 353 graphs on 27 vertices that do not contain $C_4$ or the complement of $K_{1,22}$, which took 816 seconds. We attempted to compute some of the missing entries in the table, for instance the critical graphs for $R(C_4,K_{1,17})$, but our computations timed out. It may be possible to compute these entries using SMS with additional time investment, but an alternative or more sophisticated approach is likely necessary. 

\section{Concluding remarks}

We also tried to improve the lower bounds on several other Ramsey numbers, for example $R(K_3,K_4,K_4 -e)$, $R(K_4,K_4,K_4-e), R(3,3,3,4)$, and $R(3,3,5)$, with block Cayley colorings, but in these cases we only managed to tie the best known lower bounds. It would be interesting to investigate whether other algebraic or geometric constructions can improve on the bounds given in this paper. We are optimistic, however, that the lower bounds are close to the true values. The upper bounds are generally more difficult, but we hope the gaps can be narrowed in the near future. 

\bibliography{RamseyCayley}
\bibliographystyle{abbrv}
\end{document}